\documentclass[11pt]{article}
\usepackage{amssymb,latexsym,amsmath,amsthm,times,wasysym}
\usepackage{curves,dcolumn,siunitx,booktabs,xcolor,times}
\usepackage{mathtools,amsfonts,color,rotating,bm}
\usepackage{indentfirst,enumerate, verbatim}   
\usepackage{mathrsfs}   
\usepackage{yhmath}     
\usepackage{graphicx}   
\usepackage{curves,pb-diagram,lineno}
\usepackage[mathscr]{eucal}

\newtheorem{theorem}{Theorem}[section]
\newtheorem{lemma}[theorem]{Lemma}
\newtheorem{corollary}[theorem]{Corollary}
\newtheorem{prop-def}{Proposition-Definition}[section]

\newtheorem{defi}[theorem]{Definition}

\newcommand{\q}{\quad}
\newcommand{\qq}{\qquad}

\newcommand{\ges}{\geqslant}
\newcommand{\les}{\leqslant}
\newcommand{\mb}{\mathbf}

\DeclareMathOperator{\id}{id}
\DeclareMathOperator{\RSp}{RSp}
\DeclareMathOperator{\R}{R}

\begin{document}

\title{Enumeration of Upper Triangular Symplectic\\ Renner Monoids
}
\date{}
\maketitle

\vspace{ -1.6cm}
\begin{center}
\author Reginald Koo{$^{1}$\footnotetext[1]{\em E-mail:
regk@usca.edu}} ~\mbox{and}~ Zhenheng Li{$^{2}$\footnotetext[2]{\em E-mail:
zhenhengl@usca.edu\quad Corresponding Author} }

\vskip 2mm
Department of Computer Science, Engineering and Mathematics\\
University of South Carolina Aiken\\
Aiken, SC 29801, USA
\end{center}

\vspace{ 0.3cm}
\begin{abstract}
In this paper, we compute the orders of upper triangular symplectic Renner monoids using admissible sets. We obtain both recursive and closed-form formulas for these orders. We then derive their exponential generating functions and establish connections between these functions and the Laguerre polynomials.

\vspace{ 0.3cm} \noindent {\bf Keywords:} Symplectic injective partial transformation, symplectic bijective transformation, symplectic Renner monoid, upper triangular symplectic Renner monoid.
\end{abstract}

\baselineskip 15pt
\parskip 6pt
\section{Introduction}

The set $\R_m$ of all injective partial transformations of ${\bf m}=\{1,2,\dots, m\}$ is a monoid (semigroup with an identity) under composition of maps, with identity element $\id_{\bf m}$. An {\em injective partial transformation} of {\bf m} is a one-to-one map of a subset (domain) of {\bf m} onto a subset (range) of {\bf m}. 

Let $M^*_m$ denote the set of all $m\times m$ matrices $A=[a_{ij}]$ such that each $a_{ij}=0$ or $1$  for all $i,j$, and such that each row and each column contains at most one $1$. 
Then, $M_m^*$ is a monoid under matrix multiplication with identity $I_m$. There is a monoid isomorphism  between $\R_m$ and $ M_m^*$, whereby $a_{ij}=1$ if{f} $f(j)=i$. We identify $\R_m$ with $M^*_m$ throughout. 

For instance, let 
${\bf m}=\{1, 2, 3, 4\}$. The matrix
$$
                        \begin{pmatrix}
                        0&0&1&0\\
                        1&0&0&0\\
                        0&0&0&0\\
                        0&0&0&1
                        \end{pmatrix}\in M^*_4
$$
corresponds to the injective partial transformation $f\in \R_4$ with domain $D(f)=\{1, 3, 4\}$ given by the column indices of the nonzero entries, and range $R(f)=\{1, 2, 4\}$ given by the row indices of the nonzero entries. In map notation
\[
    1\mapsto 2,\qq 2\notin D(f),\qq 3\mapsto 1,\qq4\mapsto 4.
\]

In combinatorics, these matrices can be viewed as placements of any number of non-attacking rooks on an $m\times m$ chessboard, and $\R_m$ is consequently called the {\em rook monoid}. We use the terms {\em injective partial transformation} and {\em rook matrix} interchangeably, according to the context. The submonoids $B_m$ of $\R_m$ consisting of upper triangular injective partial transformations of ${\bf m}$ are called {\em upper triangular rook monoids}.

The orders $$|\R_m| = \sum_{r=0}^m {\binom mr}^2 r!$$
are closely related to the Laguerre polynomials, which have been described as one of the ``prettiest gems'' of combinatorial theory \cite{A, BRR}. The orders of $B_m$ are the Stirling numbers of the second kind, which arise naturally in the study of set partitions. The exponential generating functions of both $|\R_m|$ and $|B_m|$ are given in \cite{BRR}.

It is sometimes useful to regard $\R_m$, where $m\ges1$, as the Renner monoid of the general linear algebraic monoid $M_m(K)$ in the theory of linear algebraic monoids \cite{ZL, PU1, R1, SL1}, where $K$ is an algebraically closed field.

The symplectic Renner monoids $\RSp_n$, where $n=m/2$ and $m$ is even, form another infinite family of rook matrices of size $2n\times2n$, consisting of symplectic injective partial transformations (Definition \ref{symplecticTransformation}). A systematic study of their structures and cardinalities using the symplectic Weyl group and admissible subsets (Definition \ref{admissibleDefinition}) can be found in \cite{CLSW, ZHL}.

The orders $|\RSp_n|$ and the exponential generating functions of the symplectic Renner monoids were studied in \cite{LLC}. In particular,
$$
  |\RSp_n|   = \sum_{r=0}^n 4^r{\binom nr}^2 r! + 2^n n!.
$$
These orders are also closely connected to the Laguerre polynomials \cite{LLC}.

Let $\mathbb B_n$ denote the subset of $\RSp_n$ consisting of upper triangular symplectic injective partial transformations. These form submonoids of $\RSp_n$, called the {\em upper triangular symplectic Renner monoids}. For $0\leq r\leq2n$, define
\begin{align*}
\mathbb B_{n,\,r}&=\{ f\in \mathbb B_n ~\big|~|D(f)|=r\},\q 0\les r\les 2n.
\end{align*}
Thus, $\mathbb B_{n,r}$ consists of the elements of $\mathbb B_n$ whose domains have cardinality $r$.

We investigate the orders and exponential generating functions of $\mathbb B_n$ and $\mathbb B_{n,r}$. In particular, we derive recursive formulas for calculating $\mathbb B_n$ and $\mathbb B_{n,r}$ in Theorem \ref{mainTheorem}. These formulas are then used to obtain the exponential generating functions for $b_n$ and $b_{n,\,n-r}$ in Theorems \ref{generatingBn} and \ref{generatingBnr}, respectively. We further establish connections between these generating functions and the Laguerre polynomials in Corollary \ref{coro:LaguerrePolynomial}. We then provide a recurrence relation of order three for the sequence $b_n$ in Theorem \ref{th:pureRecursion}. Closed formulas for $b_n$ and $b_{n,r}$ are given in Theorems \ref{formula:bn} and \ref{formula:bnr}, respectively.

The remainder of the paper is organized as follows. Section 2 introduces the concepts and techniques involving admissible sets and upper triangular symplectic Renner monoids. Section 3 develops recursive formulas for the orders of these monoids and gives numerical values for the first seven cases. Section 4 presents the exponential generating functions and closed formulas for $b_n$ and $b_{n,r}$. It also provides a recurrence relation for the sequence $b_n$ involving the preceding three terms, $b_{n-1}$, $b_{n-2}$, and $b_{n-3}$. Finally, Section 5 presents a conjecture that may be useful for further study of the generating functions of $b_n$.

\section{Preliminaries }
\vspace{-2mm}
Let ${\bf m}=\{1, 2, ..., m\}$ with $m\ges 1$. Define an involution
$\theta$ of ${\bf m}$ by 
\[
    \theta(i)= m+1-i,\quad i\in {\bf m}.
\]
Then $\theta\in\mathbf{\mathcal {\R}}_m$. 
The involution $\theta$ in $\mb{m=2n}$ can be
interpreted geometrically as reflection across the `median line', illustrated here with $n=4$:
 \begin{center}
 \begin{small}
\begin{tabular}{|c|c|c|c||c|c|c|c|}
\hline
1&2&3&4&5&6&7&8\\\hline
\end{tabular}
\end{small}
\end{center}
\begin{defi}{\label{admissibleDefinition}} 
A subset $I$ of $\mb m$ is said to be {\bf admissible} if $I=\mb m$ or if $\theta(I)\cap I=\varnothing$. 
An admissible subset not equal to $\mb m$ is called proper admissible.
\end{defi}
\vspace{-2mm}
If $m=2n$ with $n\ges 1$, we use $\mathscr A_n$ to denote the set of all admissible subsets of $\mb {2n}$. Note that the empty set $\varnothing$ is admissible.
If ${\bf m}=\{1, 2\}$, then 
$$
    \mathscr A_1 = \{\varnothing, ~\{1\}, ~\{2\},~\{1, 2\} \}.
$$ 
\vspace{-2mm}
If ${\bf m}=\{1, 2, 3, 4\}$, then 
        $$
        \mathscr A_2 =\{\varnothing, ~\{1\}, ~\{2\}, ~\{3\}, ~\{4\}, ~\{1, 2\}, ~\{1, 3\}, ~\{2, 4\}, ~\{3, 4\},~\{1, 2, 3, 4\}\}.
        $$
The two lemmas below, presented without proof, state some simple  properties of admissible sets.
\begin{lemma}{\label{lemma:intersection}}
Let ${\bf m}=\{1, 2, ..., m\}$ with $m\ges 1$. Then the intersection of two admissible sets is admissible, and each subset of a proper admissible set is admissible.
\end{lemma}

\begin{lemma}{\label{basicFacts}}
Let $m$ be even and $i,j\in \mb{m}=\{1,2,\dots, m\}$. We have
 \begin{enumerate}[\upshape a)]
      \vspace{-2mm}
 \item $\theta(i)\ne i$ and hence the set $\{i\}$ is admissible. 
      \vspace{-2mm}
 \item If $j\ne \theta (i)$, then $ \{i,j\}$ is admissible.
     \vspace{-2mm}
 \item The set $\{i,\theta(i)\}$ is not admissible.
      \vspace{-2mm}
 \item If $I$ is a proper admissible set, then $|I|\les m/2$.
 \end{enumerate}
\end{lemma}

The set $\mathscr A_n^\times = \mathscr A_n\setminus\{\mb {2n}\}$ consisting of all proper admissible subsets of $\mb {2n}$ is useful. 
The following lemma will be used in Section \ref{mainSection}.
\begin{lemma}\label{tau-lemma}
Let $n\ges 2$, let $\tau\colon\{2,\dots, 2n-1\}\to\{1,\dots, 2n-2\}$ be the map defined by $\tau(k)=k-1$, and let $\mathscr B_n=\{I\in\mathscr A_{n}^\times\mid I\cap \{1,2n\}=\varnothing\}$.
\begin{enumerate}[\rm a)]
     \vspace{-2mm}
  \item If $I\subseteq\{2,\dots, 2n-1\}$, then $I\in \mathscr B_{n} $ if and only if $\tau (I)\in \mathscr A_{n-1}^\times$.
           \vspace{-2mm}
  \item If $K\in \mathscr A_{n-1}^\times $ then both $\tau^{-1}(K)\cup\{1\}$ and $\tau^{-1}(K)\cup\{2n\}$ are in $\mathscr A_n^\times$.
\end{enumerate}
\end{lemma}
\begin{proof}
We prove a). Let $\theta$ be the involution of $\mb{2n}$, and let $\theta'$ be the involution of $\mb{2n-2}$, so that 
$\theta(i)=2n+1-i$   $(1\les i\les 2n)$, and $\theta'(k)=2n-1-k$ $(1\les k\les 2n-2)$. It is easy to verify that $\theta(i)-1=\theta'(i-1)$ for $2\les i\les 2n-1$. Let $I$ be a non-empty set in $\mathscr B_n$, and let $J=\tau(I)$.
Suppose to the contrary that $J$ is not admissible. Then there exist $i$, $j\in I$ such that $\theta'(i-1)=j-1$. It follows that
$\theta (i)=j$ contradicting the fact that $I$ is admissible. The proof
of the converse is similar.

We now prove b). If $K\in \mathscr A^\times_{n-1}$ then from a) it follows that $\tau^{-1}(K)\subseteq\{2,\dots, 2n-1\}$ is in $ \mathscr A_{n}^\times$. Since $\theta(1)$ and $\theta(2n)$ are both not in $\tau^{-1}(K)$ it follows that
$\tau^{-1}(K)\cup\{1\}$ and $\tau^{-1}(K)\cup\{2n\}$ are in $\mathscr A_n^\times$.
\end{proof}

Recall that $\R_m$ is the monoid of all injective partial transformations of ${\bf m} = \{1, 2, . . . , m\}$. The domain and range of $f\in \R_m$ are denoted by $D(f)$ and $R(f)$, respectively. 
\begin{defi}{\label{properInParTrans}}
Let $f$ be an element of $\R_m$. If $D(f)$ is the entire set $\bf m$, then $f$ is called a full transformation of $\bf m$. If $D(f)$ is a proper subset of $\bf m$, then $f$ is called a proper injective partial transformation of $\bf m$.
\end{defi}

\subsection{Symplectic Renner Monoids}

To investigate the theory of symplectic Renner monoids, we require that $m=2n$, with $n\ges 1$.

\begin{defi}{\label{symplecticTransformation}} An injective partial transformation $f$ in $\R_{2n}$ is said to be \emph{symplectic} if $D(f)$ is admissible and $f$ takes admissible subsets of $D(f)$ to admissible subsets of $\bf 2n$. The set of all symplectic transformations is denoted by
\[
    \begin{aligned}
        \RSp_n   &= \{f\in \R_{2n} \mid f \mbox{ is symplectic}\}.
    \end{aligned}
\]
\end{defi}

The identity element $\id_{\bf 2n}$ of $\R_{2n}$ and the involution $\theta$ lie in $\RSp_{n}$.
Note that a full transformation in $\R_{2n}$  is symplectic if it takes admissible subsets of $\bf 2n$ to admissible subsets of $\bf 2n$, that is, it preserves $\mathscr{A}_n$. It is evident that $\id_{\bf 2n}$ is full symplectic. Furthermore, $\theta$ is full symplectic. First note that $\theta$ is bijective. Next if $I$ is proper admissible in $\mb{2n}$ and $J=\theta(I)$, then $J\cap\theta(J)=\theta(I)\cap I=\varnothing$. Thus, $\theta(I)$ is admissible.  

The full injective transformation
\[
h=\begin{pmatrix}
1&0&0&0\\
0&1&0&0\\
0&0&0&1\\
0&0&1&0
\end{pmatrix}
\]
is not symplectic, since $I=\{2,4\}$ is admissible yet $h(I)=\{2,3\}$ is not admissible. 

Combining Lemmas \ref{lemma:intersection} and \ref{basicFacts} d), we obtain the following result, which will be used implicitly throughout the rest of the paper.
\begin{lemma}{\label{lemma:properSymplectic}}
  \begin{enumerate}[\upshape a)]
    \item A proper injective partial transformation is symplectic if{f} both its domain and range are admissible.
        \vspace{-2mm}
    \item There are no proper symplectic injective partial transformations with domain size between $n+1$ and $2n-1$ inclusive.
  \end{enumerate}
\end{lemma}

An element $f$ of a monoid $(M,\circ, e)$ is said to be a {\bf unit} if there exists $g\in M$ such that $fg=gf=e$, the identity element of $M$. The set of the units of $M$ is a group, called the \textbf{unit group} of $M.$ 

The group $S_{2n}$ of all bijective transformations of $\bf {2n}$ is the unit group of $\R_{2n}$. In fact, $S_{2n}$ consists precisely of the non-singular rook matrices of size ${2n}\times {2n}.$
Both $S_{2n}$ and $\R_{2n}$ share the same identity element $\id_{\mb{2n}}$.

Let $f\in\R_{2n}$. Its (partial) inverse, denoted by $f^{-1}$, is defined by:
\[
    f^{-1}(b) = a \quad\text{if}\quad f(a)=b, \quad\text{where } a\in D(f), b\in R(f).
\]
If $f$ is full, then $f\in S_{2n}$ and its partial inverse $f^{-1}$ is the inverse of $f$ in $S_{2n}$. 

  \vspace{2mm}
\begin{lemma}{\label{Rspn}}
\begin{enumerate}[\upshape a)]
  \item If $f\in \RSp_n$, then $f^{-1}\in \RSp_n$.
  \vspace{-2mm}
  \item If 
   $f,g\in \RSp_n$, then $gf\in \RSp_n$. 
   \end{enumerate}
\noindent
Thus, $\RSp_n$ is a submonoid of $\R_{2n}$.
\end{lemma}
\begin{proof} 
 a)~ Let $f$ be symplectic. If $f$ is full, then $f$ preserves admissible sets, hence all powers $f^k$ $( k\ges 1)$ are symplectic. Note that $f$ is in $S_{2n}$; we deduce that $f^{-1}$ is symplectic, since $f^{-1}$ is a power of $f$.
 If $f$ is proper, then both $D(f)$ and $R(f)$ are proper admissible. But then $D(f^{-1})=R(f)$ and
 $R(f^{-1})=D(f)$. Thus, $f^{-1}$ is symplectic.
  
b)~ Let $f$, $g$ be symplectic and $h=gf$, where $f$ is applied first. Suppose both $f$ and $ g$  are full. Then $h$ is full, and since $f$ and $g$ preserve admissible sets, $h$ is symplectic. Next, suppose at least one of $f$ and $g$ is proper. Then $h$ is proper, and
 $K=R(f)\cap D(g)$ is proper admissible, whence $D(h)=f^{-1}(K)$ is proper admissible by a), and $R(h)=g(K)$ is proper admissible. Thus, $h\in \RSp_n$. We note that if $K$ is empty then
 $h$ is the zero matrix.
 
Since $\id_{2n}\in \RSp_n$, it follows that $\RSp_n$ is a submonoid of $R_{2n}$.
\end{proof} 
The monoids $\RSp_n$ are called the symplectic Renner monoids in the theory of reductive algebraic monoids. See \cite{CLSW, LLC, ZHL, ZL} for more details. 
\begin{lemma} Let
  $W_n:=\RSp_n\cap S_{2n}$. Then $W_n$ is the unit group of $\RSp_n$. 
\end{lemma}
\begin{proof} 
Write $U$ for the unit group of $\RSp_n$. 
If $f\in U$ then $f\in \RSp_n$ and there exists $g\in \RSp_n$ such that $fg=gf=\id_{2n}$. Hence $f\in S_{2n}$, whence $f\in RSp_n\cap S_{2n}= W_n$.
On the other hand suppose that $f\in W_n$. Then $f\in \RSp_n$. Since also $f\in S_{2n}$,  there is $g\in S_{2n}$
  such that $fg=gf=\id_{2n}$. By  Lemma \ref{Rspn} a), $g\in \RSp_n$. Thus $f\in  U$. 
\end{proof}
To give a more concrete characterization of $W_n$, we can describe its elements in terms of their action on admissible sets: 
  \begin{align*}
     W_n &= \{f\in S_{2n}\mid f \text{ takes admissible sets to admissible sets}\}\\
        &= \{f\in\R_{2n}\mid f \text{ is full symplectic}\}.
  \end{align*}
We then have the following disjoint union.
\[
    \begin{aligned}
        \RSp_n   &= W_n \sqcup\{f\in \R_{2n} \mid D(f) \mbox{ and }R(f) \mbox{ are proper adimissible}\}.
    \end{aligned}
\]
Thus
$\RSp_n$ consists of the elements in its unit group $W_n$ and the proper symplectic injective partial transformations. 

If $n=1$, then 
$$
     \RSp_1 = \left\{  \begin{pmatrix}
                        0&0\\
                        0&0
                        \end{pmatrix},
                        \begin{pmatrix}
                        1&0\\
                        0&0
                        \end{pmatrix},
                        \begin{pmatrix}
                        0&1\\
                        0&0
                        \end{pmatrix},
                        \begin{pmatrix}
                        0&0\\
                        1&0
                        \end{pmatrix},
                        \begin{pmatrix}
                        0&0\\
                        0&1
                        \end{pmatrix},
                        \begin{pmatrix}
                        1&0\\
                        0&1
                        \end{pmatrix},
                        \begin{pmatrix}
                        0&1\\
                        1&0
                        \end{pmatrix}
                        \right\}.
$$

If $n=2$, then $\RSp_2$ consists of fifty seven elements \cite{ZHL}. The following proper injective partial transformation $f$ lies in $\RSp_2$, whereas $g$ doesn't lie in $\RSp_2$:
%
\[\label{fAndg}
\small{
               f =   \begin{pmatrix}
                        1&0&0&0\\
                        0&0&0&0\\
                        0&1&0&0\\
                        0&0&0&0\\
                    \end{pmatrix}
                    \quad\mbox{and}\quad
                    g =   \begin{pmatrix}
                        0&1&0&0\\
                        0&0&0&0\\
                        0&0&1&0\\
                        0&0&0&0\\
                    \end{pmatrix}.
    }
\]
This is because $D(f)=\{1, 2\}$ and $R(f) = \{1, 3\}$ are admissible, whereas $D(g)=\{2, 3\}$ is not admissible although the range $R(g) = \{1, 3\}$ is.

\begin{lemma}{\label{SymWeylGroup}}
The unit group $W_n$ of $\RSp_n$ is the centralizer of $\theta$ in $S_{2n}$. In other words, $W_n$ equals
\[
    C_{S_{2n}}(\theta) = \{f\in S_{2n} \mid \theta f = f\theta\}.
\] 
\end{lemma}
\begin{proof}
We note that  $\theta\in W_n$ since $\theta\in RSp_n\cap S_{2n}$. 
 Denote $C=C_{S_{2n}}(\theta)$. 

We first show that $C\subseteq W_n$. 
Let $f\in C$. Then $f\in S_{2n}$, thus $D(f)=\bf 2n$, so it suffices to show that for any proper admissible subset $A$ of $\mb{2n}$ we have $f(A)$ is admissible. Suppose for a contradiction that $\theta(f(A))\cap f(A)$ is not empty.
Then there exist $i$, $j$ in $A$ such that $\theta (f(i))=f(j)$. As $f\in C$, it follows that $f(\theta (i))=f(j)$, 
whence $\theta(i)=j$ since $f$ is injective. Since $A$ is admissible, this is a contradiction.

We next show that $W_n\subseteq C.$ Suppose for a contradiction that there is $f\in W_n$ but $f\notin C.$ 
Then there is $i\in{\bf 2n}$ such that $f(\theta(i))\ne \theta(f(i)).$ It follows from Lemma \ref{basicFacts} b) that the set $\{f(i), f(\theta(i))\}$ is admissible. Since $f^{-1}\in W_n$, the set $f^{-1}(\{f(i), f(\theta(i)\}) = \{i, \theta(i)\}$ is admissible, which contradicts Lemma \ref{basicFacts} c). 
\end{proof}

It turns out that $W_n$ is isomorphic to the symplectic Weyl group in Lie theory. We sometimes just say that $W_n$ is the symplectic Weyl group.  If $n=1$, then $\theta = (12)$ and 
$W_1 =C_{S_2}(\theta)=\{(1), (12)\}.$ The matrix form is
$$W_1=\left\{\begin{pmatrix}
                        1&0\\
                        0&1
                        \end{pmatrix},
                        \begin{pmatrix}
                        0&1\\
                        1&0
                        \end{pmatrix}
                        \right\}.
$$
If $n=2$, then $\theta = (14)(23)$ and 

$W_2 =C_{S_4}(\theta)=\{(1), (14), (23), (14)(23), (12)(34), (1243), (1342), (13)(24)\}
$

$\hspace{5mm} = 
\left\{\scriptsize  
              \begin{pmatrix}
                        1&0&0&0\\
                        0&1&0&0\\
                        0&0&1&0\\
                        0&0&0&1\\
                \end{pmatrix},
                \begin{pmatrix}
                        0&0&0&1\\
                        0&1&0&0\\
                        0&0&1&0\\
                        1&0&0&0\\
                \end{pmatrix},
                \begin{pmatrix}
                        1&0&0&0\\
                        0&0&1&0\\
                        0&1&0&0\\
                        0&0&0&1\\
                \end{pmatrix},
                \begin{pmatrix}
                        0&0&0&1\\
                        0&0&1&0\\
                        0&1&0&0\\
                        1&0&0&0\\
                \end{pmatrix},\right.$
              
        $ \hspace{12mm}      \left.\scriptsize \begin{pmatrix}
                        0&1&0&0\\
                        1&0&0&0\\
                        0&0&0&1\\
                        0&0&1&0\\
                \end{pmatrix},
                \begin{pmatrix}
                        0&0&1&0\\
                        1&0&0&0\\
                        0&0&0&1\\
                        0&1&0&0\\
                \end{pmatrix},
                \begin{pmatrix}
                        0&1&0&0\\
                        0&0&0&1\\
                        1&0&0&0\\
                        0&0&1&0\\
                \end{pmatrix},
                \begin{pmatrix}
                        0&0&1&0\\
                        0&0&0&1\\
                        1&0&0&0\\
                        0&1&0&0\\
                \end{pmatrix}
\right\}.
$

\subsection{Upper Triangular Symplectic Renner Monoids $\mathbb B_n$}
Let $\mathbb B_n$ be the submonoid of $\RSp_n$ consisting of all the upper triangular symplectic injective partial transformations. That is,
\[
    \mathbb B_n = \{f\in \RSp_n \mid f \mbox{ is upper triangular}\}.
\]
It is evident that the unit group of $\mathbb B_n$ is the trivial
group $\{\id_{\mb{2n}}\}$.

\begin{defi}
The submonoid $\mathbb B_n$ of $\RSp_n$ is called the upper triangular symplectic Renner monoid. 
\end{defi}
If $n=1,$ then
$$
     \mathbb B_1 = \left\{  \begin{pmatrix}
                        0&0\\
                        0&0
                        \end{pmatrix},
                        \begin{pmatrix}
                        1&0\\
                        0&0
                        \end{pmatrix},
                        \begin{pmatrix}
                        0&1\\
                        0&0
                        \end{pmatrix},
                        \begin{pmatrix}
                        0&0\\
                        0&1
                        \end{pmatrix},
                        \begin{pmatrix}
                        1&0\\
                        0&1
                        \end{pmatrix}  
                      \right\}.
$$

If $n=2,$ then

$
\mathbb B_2 = 
\left\{  \scriptsize{
              \begin{pmatrix}
                        0&0&0&0\\
                        0&0&0&0\\
                        0&0&0&0\\
                        0&0&0&0\\
                \end{pmatrix},
                \begin{pmatrix}
                        1&0&0&0\\
                        0&0&0&0\\
                        0&0&0&0\\
                        0&0&0&0\\
                        \end{pmatrix},  
                        \begin{pmatrix}
                        0&1&0&0\\
                        0&0&0&0\\
                        0&0&0&0\\
                        0&0&0&0\\
                        \end{pmatrix},
                        \begin{pmatrix}
                        0&0&1&0\\
                        0&0&0&0\\
                        0&0&0&0\\
                        0&0&0&0\\
                        \end{pmatrix},}\right.$ \\[4pt]
 
\vspace{-2mm}                       
$ \hspace{13mm} \scriptsize{           \begin{pmatrix}
                        0&0&0&1\\
                        0&0&0&0\\
                        0&0&0&0\\
                        0&0&0&0\\
                        \end{pmatrix},
                        \begin{pmatrix}
                        0&0&0&0\\
                        0&1&0&0\\
                        0&0&0&0\\
                        0&0&0&0\\
                        \end{pmatrix},
                        \begin{pmatrix}
                        0&0&0&0\\
                        0&0&1&0\\
                        0&0&0&0\\
                        0&0&0&0\\
                        \end{pmatrix},
                        \begin{pmatrix}
                        0&0&0&0\\
                        0&0&0&1\\
                        0&0&0&0\\
                        0&0&0&0\\
                        \end{pmatrix},}$ \\[4pt]
      
      \vspace{-2mm}                  
$  \hspace{13mm}  \scriptsize{           \begin{pmatrix}
                        0&0&0&0\\
                        0&0&0&0\\
                        0&0&1&0\\
                        0&0&0&0\\
                        \end{pmatrix},
                        \begin{pmatrix}
                        0&0&0&0\\
                        0&0&0&0\\
                        0&0&0&1\\
                        0&0&0&0\\
                        \end{pmatrix},
                        \begin{pmatrix}
                        0&0&0&0\\
                        0&0&0&0\\
                        0&0&0&0\\
                        0&0&0&1\\
                        \end{pmatrix},
                   \begin{pmatrix}
                        1&0&0&0\\
                        0&1&0&0\\
                        0&0&0&0\\
                        0&0&0&0\\
                        \end{pmatrix},}$\\[4pt]
             
             \vspace{-2mm}           
$  \hspace{13mm} \scriptsize{        \begin{pmatrix}
                        1&0&0&0\\
                        0&0&1&0\\
                        0&0&0&0\\
                        0&0&0&0\\
                        \end{pmatrix},
                        \begin{pmatrix}
                        1&0&0&0\\
                        0&0&0&0\\
                        0&0&1&0\\
                        0&0&0&0\\
                        \end{pmatrix},
                        \begin{pmatrix}
                        0&1&0&0\\
                        0&0&0&1\\
                        0&0&0&0\\
                        0&0&0&0\\
                        \end{pmatrix},
                        \begin{pmatrix}
                        0&1&0&0\\
                        0&0&0&0\\
                        0&0&0&1\\
                        0&0&0&0\\
                        \end{pmatrix},}$\\[4pt]
                   
                   \vspace{-2mm}                         
$ \hspace{13mm} \scriptsize{         \begin{pmatrix}
                        0&0&1&0\\
                        0&0&0&1\\
                        0&0&0&0\\
                        0&0&0&0\\
                        \end{pmatrix},
                       \begin{pmatrix}
                        0&0&1&0\\
                        0&0&0&0\\
                        0&0&0&1\\
                        0&0&0&0\\
                        \end{pmatrix},
                        \begin{pmatrix}
                        0&0&0&1\\
                        0&0&0&0\\
                        0&0&1&0\\
                        0&0&0&0\\
                        \end{pmatrix},
                        \begin{pmatrix}
                        0&0&0&1\\
                        0&0&1&0\\
                        0&0&0&0\\
                        0&0&0&0\\
                        \end{pmatrix},}$\\[4pt]
         
         \vspace{-2mm}               
$ \hspace{13mm}  \scriptsize{        \begin{pmatrix}
                        0&0&0&1\\
                        0&1&0&0\\
                        0&0&0&0\\
                        0&0&0&0\\
                        \end{pmatrix},
                      \begin{pmatrix}
                        0&0&0&0\\
                        0&1&0&0\\
                        0&0&0&0\\
                        0&0&0&1\\
                        \end{pmatrix},
                        \begin{pmatrix}
                        0&0&0&0\\
                        0&0&1&0\\
                        0&0&0&0\\
                        0&0&0&1\\
                        \end{pmatrix},
                        \begin{pmatrix}
                        0&0&0&0\\
                        0&0&0&0\\
                        0&0&1&0\\
                        0&0&0&1\\
                        \end{pmatrix},}$\\[4pt]
                        
                \vspace{-2mm}        
$\hspace{13mm}\left. \scriptsize{          \begin{pmatrix}
                        1&0&0&0\\
                        0&1&0&0\\
                        0&0&1&0\\
                        0&0&0&1\\
                        \end{pmatrix}   }   
    \right\}.$ 

\noindent
We then have the orders $|\mathbb B_1| = 5$ and $|\mathbb B_2| = 25$.

\section{Orders of $\mathbb B_n$}\label{mainSection}

Let ${\bf 2n}=\{1, 2, ..., 2n\}$ and $n\ges 1$. Recall that members of $\mathbb B_n$ are upper triangular symplectic rook matrices of size $2n\times 2n$.
\begin{defi} Let $n\ges 1$. We define
\begin{align*}
\mathbb B_{n,\,r}&=\{ f\in \mathbb B_n ~\big|~|D(f)|=r\},\q 0\les r\les 2n.
\end{align*}
\end{defi}
By Lemma \ref{lemma:properSymplectic} b), $\mathbb B_{n,\,r}=\varnothing$ for $n+1\les r\les 2n-1$, and $\mathbb B_{n,\,2n}=\{\id_{2n}\}$ since the identity element $\id_{\mb{2n}}$ is the only element in $\mathbb B_n$ of rank $2n$. 
It follows immediately that
\begin{align*}
    \mathbb{B}_n &= \bigsqcup_{r=0}^{2n} \mathbb{B}_{n,\, r}= \Big(\bigsqcup_{r=0}^{n} \mathbb{B}_{n,\, r}\Big)\bigsqcup\,\{\id_{2n}\}.
\end{align*}

Define $b_{n,\,r}=|\mathbb B_{n,\,r}|$ for $0\les r\les n$. Then $b_{n,\,0}=1$. We agree that $b_{0,\,0}=1$ and $b_{n,\,r}=0$ if $n<r$ or $r<0$.  
The following lemmas are useful for calculating the order of $\mathbb{B}_n.$ Write $\bar{i}=\theta(i)$. Then $\bar 1=2n.$ 
\begin{lemma}\label{lemma-Xnr1} Let $n\ges 1$ and  $0\les r\les n$. Write
\begin{align*}
    X_{n,\,r}&=\{f\in \mathbb B_{n,\,r}\mid 1\notin R(f)\text{ and } \bar{1}\notin D(f)\}.
\end{align*}
%
Then there exists a bijection of $X_{n,\,r}$ onto $\mathbb B_{n-1,\,r}$. Thus, $|X_{n,\,r}|=b_{n-1,\,r}$.
\end{lemma}
\begin{proof} 
If $n=1$, then $r=0, 1$. Thus, we have
 \begin{align*}
 &X_{1,\,0}=\{0\},\q |X_{1,\,0}|=1;\q b_{0,\,0}=1\\
 &X_{1,\,1}=\varnothing,\q~~~ |X_{1,\,1}|=0; \q b_{0,\,1}=0.
 \end{align*}
Thus, the lemma holds for $0\les r\les n=1.$

We now suppose that $n\ges 2$.  Let $0\les r\les n$. 
If $f\in X_{n,\,r}$, then row $1$ and column $\bar{1}$  are both zero and also row $\bar{1}$ and column 1 are zero. Therefore, $f$ takes the form
\[
\begin{pmatrix}
            0 & 0 & 0 &\cdots & 0& 0 \\
              & * & * &\cdots & *& 0  \\
              &   & * &\cdots & *& 0 \\
              &   &   &\ddots &\vdots&\vdots \\
              &   &   &       & *& 0  \\
              &   &   &       &  & 0  \\
             \end{pmatrix}_{2n \times 2n} \in X_{n,\,r}~,
\]
where $*$ is either 0 or 1, each row and each column of $f$ contains at most one $1$, and $D(f)$ and $R(f)$ are of cardinality $r$ and proper admissible. 

If $r=n$, then by Lemma \ref{basicFacts} d) and Lemma \ref{tau-lemma} we have 
\[
    X_{n,\,n}=\varnothing,\q |X_{n,\,n}|=0,\q b_{n-1,\,n}=0.
\]  
That is, the result is true for $r=n$.

Suppose now that $0\les r\les n-1$, and again $n\ges 2$.
Let $f'$ be obtained from $f\in X_{n,\,r}$ by deleting row $1$, row $\bar{1}$, column $1$, and column $\bar{1}$. If $f=[a_{ij}]\in X_{n,\,r}$, then $f'=[b_{k\,l}]$ where $b_{k\,l}=a_{k+1,\,l+1}$.
By Lemma \ref{tau-lemma} a) both the domain and range of $f'$ are admissible sets in $\mathscr A_{n-1}$. Hence $f'\in \mathbb B_{n-1,\,r}$. Then we have a map
$\kappa: f\mapsto f'$. The map $\kappa$ is onto:  given $g\in \mathbb B_{n-1,\,r}$, then adjoining two zero rows and two zero columns on the outside of $g$ yields an $f\in X_{n,\,r}$ for which $f'=g$.
The map $\kappa$ is one-to-one: if $f,\,g\in X_{n,\,r}$ and $f'= g'$, then  $f=g$
since $f$ indeed agrees with $g$ on row  $1$, row $\bar{1}$, column $1$, and column $\bar{1}$. Hence $|X_{n,\,r}|=b_{n-1,\,r}$.
\end{proof}
\begin{lemma}\label{prop-Xnr2}
Let $n\ges 1$ and $0\les r\les n$. Define
\[
Z_{n,\,r}=\{f\in\mathbb B_{n,\,r}\mid 
1\notin R(f)\text{ and } f(\bar 1)=j \;\text{ for some $2\les j\les 2n-1$}\}.
\]
%
Then 
$$
    |Z_{n,\,r}|=(2n-2r)b_{n-1,\,r-1}.
$$
\end{lemma}
\begin{proof}
Take $n=1$. We have
\begin{align*}
&n=1,\,r=0:\q Z_{n,\,r}=Z_{1,\,0}=\varnothing;
&& (2n-2r)b_{n-1,\,r-1}=2b_{0,\,-1}=0\\
&n=1,\,r=1:\q Z_{n,\,r}=Z_{1,\,1}=\varnothing;&&
(2n-2r)b_{n-1,\,r-1}=0b_{0,\,0}=0.
\end{align*} 
The lemma holds when $n=1$.

We now suppose that $n\ges 2$. If $r=0$, then 
$Z_{n,\,r}=Z_{n,\,0}=\varnothing$ whereas $(2n-2r)b_{n-1,\,r-1}=2nb_{n-1,\,-1}=0$,\, so the result is true when $r=0$.
If $r=n$, then  $Z_{n,\,r}=Z_{n,\,n}=\varnothing$,\, whereas 
$(2n-2r)b_{n-1,\,r-1}=0$,\, so the result is true when $r=n$. 

We may now assume that $1\les r\les n-1$ so that $r-1\ges 0$ and $n-1\ges 1$.
If $f\in Z_{n,\,r}$, then $f$ is of the form
\[
\begin{array}{l@{}l}
\begin{pmatrix}
            0 & 0 & \cdots & 0&\cdots & 0& 0 \\
              & * & \cdots & *&\cdots & *& 0 \\
              &   & \ddots &\vdots&\cdots & \vdots& \vdots \\
              &   &        &0 & \cdots& 0& 1 \\
              &   &        &  &\ddots &\vdots&\vdots \\
              &   &        &  &       & *& 0  \\
              &   &        &  &       &  & 0  \\
             \end{pmatrix}_{2n \times 2n}  
   & \hspace{-2em}{\leftarrow\;}{\scriptsize a_{j,\,2n}=1}~,
   \end{array}
\]
where $2\les j\les 2n-1$. Deleting the first row, first column, last row, and last column of $f$ produces an element $f'\in\mathbb B_{n-1,\,r-1}$. 

We shall show that for each $g\in \mathbb B_{n-1,\,r-1}$, there are $2n-2r$ distinct $f$'s in $Z_{n,\,r}$ such that $f'=g$.  Indeed fix $g\in \mathbb B_{n-1,\,r-1}$, 
where $0\les r-1\les n-2$. Let $J=\tau^{-1}(R(g))$, where $\tau(k)=k-1$.
By Lemma \ref{tau-lemma} a), $J$ is admissible in $\mathscr A_n$, and $|J\cup \theta(J)|=2r-2$, where $\theta$ is the involution 
in $\mb {2n}$. Let $E(g)$ denote the complement of $J\cup\theta(J)$ in 
$\{2,\dots, 2n-1\}$, so that $|E(g)|=2n-2r\ge 2$.

For each  $j\in E(g)$ we extend $g$ by adjoining a zero row at the top and the bottom, a zero column at the left, and a column with $1$ in position $j$ at the right; this extension is denoted by $\hat{g}_j$. The domain
$D(\hat{g}_j)=\tau^{-1}(D(g))\cup\{2n\}$ is admissible in $\mathscr A_n$ by Lemma \ref{tau-lemma} b). The range 
$R(\hat{g}_j)=\tau^{-1}(R(g)) \cup\{j\}$
is also admissible since $\theta(R(\hat g_j))\cap R(\hat g_j)$ is empty, by the construction of $E(g)$ and that $\theta(j)\ne j$. 
Thus $ \hat g_j$ is in $Z_{n,\,r}$. Moreover $\kappa(\hat g_{j})=g$ for all $j\in E(g)$.

Next, note that if $j\ne j'$ in $E(g)$, then $ \hat g_j\ne \hat g_{j'}$.
Moreover, if $g\ne h$ in $\mathbb B_{n-1,\,r-1}$, then for each $j\in E(g)$ and $k\in E(h)$ we have
$ \hat g_j\ne  \hat h_k$.
Furthermore, if $f\in Z_{n,\,r}$, then by restriction we obtain an upper triangular symplectic injective partial transformation  $g$  in  $\mathbb B_{n-1,\,r-1}$ 
 in which $\hat g_j=f$ when $j=f(2n)$.

Thus $Z_{n,\,r}$ is a disjoint union of sets $\{ \hat g_j \mid j\in E(g)\}$ as $g$ runs over $\mathbb B_{n-1,\,r-1}$. 
Hence
\[
    |Z_{n,\,r}|=(2n-2r)b_{n-1,\,r-1}.
\]
The proof is complete. 
\end{proof}

\begin{lemma}\label{anti-diag-lemma}
Let $n\ges 1$, let $f\in \RSp_n$ and let $g$ be its anti-transpose. Then $f\in \mathbb B_{n,\,r}$  if{f} $g\in \mathbb B_{n,\,r}$ for $0\les r\les n$.
\end{lemma}
\begin{proof}
  Write $f=[a_{i\,j}].$   
  Clearly, $f$ is upper triangular if{f} its anti-transpose $g=[a_{\theta(j)\;\theta(i)}]$ is upper triangular. Note that $D(g)=R(f)$ and $R(g)=D(f)$. It follows that $f$ is symplectic if{f} $g$ is symplectic. The desired result follows.
\end{proof}

We describe the order of the upper triangular symplectic Renner monoid $\mathbb B_n$. 
Let $n\ges 1$ and $b_n=|\mathbb B_n|$, and we agree that $b_{0}=2$.
Recall that $b_{n,\,r}=|\mathbb B_{n,\,r}|$ for $0\les r\les n$. Then $b_n=1+\sum\limits_{r=0}^nb_{n,r}$ for $n\ges 0.$ Note that $b_{0,\,0}=1$ and $b_{n,\,r}=0$ if $r>n$ or $r<0$.  
\begin{theorem} {\label{mainTheorem}}
Let $n\ge 1$. Then
\begin{enumerate}[\upshape 1)]
  \item $b_{n, \,r} = b_{n-1, \,r} + (4n-4r+3)b_{n-1, ~r-1} + 4(n-r+1)^2 b_{n-1, ~r-2},\q\mbox{ for } n\ges 1 \text{ and } 0\les r\les n.$
  \item $b_n = 4b_{n-1} + 4\big[\sum_{r=1}^{n-1} (n-r)(n-r+1)b_{n-1, \,r-1}\big]-3,\q\mbox{ for }~ n\ges 1.$
  \item $b_n=1+\sum_{r=0}^{n-1}4\big[1+(n-r)(n-r-1)\big]b_{n-1,\,r},\q\mbox{ for }~ n\ges 1.$
\end{enumerate}
\end{theorem}

\begin{proof}

We prove 1). 
If $n\ges 1$ and $r=0$, then 1) is valid since
\begin{align*}
&b_{n,\,0}=1,\q b_{n-1,\,0}=1, \;b_{n-1,\,-1}=0=b_{n-1,\,-2}.
\end{align*}
We now suppose that $1\les r\les n$. Thus, $\mathbb B_{n,\,r}= X_{n,\,r}\sqcup Y_{n,\,r}$ where 
\[
X_{n,\,r}=\{f\in \mathbb B_{n,\,r}\mid 1\notin R(f)\}\text{\quad and\quad}
Y_{n,\,r}=\{f\in \mathbb B_{n,\,r}\mid 1\in R(f)\}.
\]

We first calculate $|X_{n,\,r}|$. If $f\in X_{n,\,r}$, then the first row of $f$ is zero since $1\notin R(f)$, and hence $f$ takes the form
\[
\begin{pmatrix}
            0 & 0 & 0 &\cdots & 0& 0 \\
              & * & * &\cdots & *& *  \\
              &   & * &\cdots & *& * \\
              &   &   &\ddots &\vdots&\vdots \\
              &   &   &       & *& *  \\
              &   &   &       &  & *  \\
             \end{pmatrix}_{2n \times 2n} \in X_{n,\,r}~,
\]
Furthermore, if $\bar 1\in D(f)$, then there exists exactly one $j$ such that $f(\bar{1})=j$, where $j\ne 1$ since again $1\notin R(f)$, so that $2\les j\les 2n$. 
Partition $X_{n,\,r}=X_{n,\,r}^1\sqcup X_{n,\,r}^2\sqcup X_{n,\,r}^3$ where
\begin{align*}
X_{n,\,r}^1&=\{f\in \mathbb B_{n,\,r} \mid 1\notin R(f)\text{ and } \bar 1\notin D(f)\}.\\
X_{n,\,r}^2&=\{f\in \mathbb B_{n,\,r} \mid 1\notin R(f)\text{ and } f(\bar{1})=j \text{ for some } 2\les j\les 2n-1\}.\\
X_{n,\,r}^3&=\{f\in \mathbb B_{n,\,r}\mid 1\notin R(f) \text{ and } f(\bar{1})=2n\}.
\end{align*}
By Lemma \ref{lemma-Xnr1}, we obtain $$|X_{n,\,r}^1|=b_{n-1,\,r}.$$
From Lemma \ref{prop-Xnr2} it follows that  
$$|X_{n,\,r}^2|=(2n-2r)b_{n-1,\,r-1}.$$
If $f\in X_{n,\,r}^3$, then its first row is zero and $f(\bar{1})=2n$. Hence $f$ is of the form
\[
\begin{pmatrix}
            0 & 0 & 0 &\cdots & 0& 0 \\
              & * & * &\cdots & *& 0  \\
              &   & * &\cdots & *& 0 \\
              &   &   &\ddots &\vdots&\vdots \\
              &   &   &       & *& 0  \\
              &   &   &       &  & 1  \\
             \end{pmatrix}_{2n \times 2n}\in X_{n,\,r}^3~.
\]
A similar argument to that in Lemma $\ref{lemma-Xnr1}$ yields
\[
    |X_{n,\,r}^3|=b_{n-1,\,r-1}.
\]
From $|X_{n,\,r}| = |X_{n,\,r}^1|+|X_{n,\,r}^2|+|X_{n,\,r}^3|$, it follows that
\[\label{Xnr}
   |X_{n,\,r}| = b_{n-1, \,r} + (2n-2r+1) b_{n-1,\, ~r-1}, \q 1\les r\les n.\tag{1}
\]

We now calculate the cardinality of $Y_{n,\,r} = \{f\in \mathbb B_{n,\,r}\mid 1\in R(f)\}$, where again $1\les r\les n$.
If $f\in Y_{n,\,r}$, then there is exactly one $1$ in the first row of $f$. 
By the admissibility of $R(f)$, the last row of $f$ must be zero. 
Therefore, $f$ is of the form
\[
\overset{\text{column } i}{\overset{\downarrow}{
\begin{pmatrix}
            0 & 0 & \cdots & 1&\cdots & 0& 0 \\
              & * & \cdots & 0&\cdots & *& * \\
              &   & \ddots &\vdots&\cdots & \vdots& \vdots \\
              &   &        &0 & \cdots& *& * \\
              &   &        &  &\ddots &\vdots&\vdots \\
              &   &        &  &       & *& *  \\
              &   &        &  &       &  & 0  \\
             \end{pmatrix}}}_{2n \times 2n}\in Y_{n,\,r}~,
\]
where $*$ is either 0 or 1, each row and each column of $f$ contains at most one $1$, and $|D(f)|=|R(f)|=r$.  

If  $f\in Y^4_{n,\,r}$ we have the following four exhaustive and mutually disjoint cases:

a) $\bar{1}\notin D(f)$ and $f(1)=1$, where the latter implies the former by the admissibility of $D(f)$.

\vspace{-1mm}
b) $\bar{1}\notin D(f)$ and $f(i)=1$ for some $2\les i\les 2n-1$. 

\vspace{-1mm}
c) $\bar{1}\in D(f)$ and $f(\bar 1)=1$.

\vspace{-1mm}
d) $\bar{1}\in D(f)$ and $f(\bar 1)= j$ for some $2\les j\les 2n-1$. 

\noindent
We can divide $Y_{n,\,r}$ into four disjoint subsets. 
Let
\begin{align*}
    Y_{n,\,r}^1 &= \{f\in \mathbb B_{n,\,r}\mid f(1)=1\}.   \\
    Y_{n,\,r}^2 &= \{f\in \mathbb B_{n,\,r}\mid f(i)=1 \text{ for some } 2\les i\les 2n-1 \text{ and }\bar{1}\notin D(f) \}.   \\
    Y_{n,\,r}^3 &= \{f\in \mathbb B_{n,\,r}\mid f(\bar{1})=1\}.\\
    Y_{n,\,r}^4 &= \{f\in \mathbb B_{n,\,r}\mid f(i)=1 \text{ and } f(\bar{1})=j \text{ for some } 2\les i, \,j\les 2n-1\}.
\end{align*}
Then $Y_{n,\,r}= Y_{n,\,r}^1\sqcup Y_{n,\,r}^2\sqcup Y_{n,\,r}^3\sqcup Y_{n,\,r}^4$. 

We first calculate $|Y_{n,\,r}^1|$. Let $f\in Y_{n,\,r}^1$. Then $f(1)=1$ and $\bar{1}\notin D(f)$. Thus $f$ takes the form
\[
\begin{pmatrix}
            1 & 0 & 0 &\cdots & 0& 0 \\
              & * & * &\cdots & *& 0  \\
              &   & * &\cdots & *& 0 \\
              &   &   &\ddots &\vdots&\vdots \\
              &   &   &       & *& 0  \\
              &   &   &       &  & 0  \\
             \end{pmatrix}_{2n \times 2n}\in Y_{n,\,r}^1~.
\]
The matrices in $Y_{n,\,r}^1$ are exactly the anti-transposes of those in 
$X_{n,\,r}^3$. 
From Lemma $\ref{anti-diag-lemma}$ we obtain
\[\label{Ynr1}
   |Y_{n,\,r}^1| = |X_{n,\,r}^3| = b_{n-1, \,r-1},\quad 1\les r\les n.\tag{2}
\]

We then calculate $|Y_{n,\,r}^2|$. First assume that $n\ges 2.$ Let $f\in Y_{n,\,r}^2$. Then
 $f(i)=1$ for some $i \in \{2, \dots, 2n-1\}$ and $f(\bar{1})$ is undefined. Therefore, the last column of $f$ is zero, and $f$ is of the form
\[
\overset{\text{column } i}{\overset{\downarrow}{
\begin{pmatrix}
            0 & 0 & \cdots & 1&\cdots & 0& 0 \\
              & * & \cdots & 0&\cdots & *& 0 \\
              &   & \ddots &\vdots&\cdots & \vdots& \vdots \\
              &   &        &0 & \cdots& *& 0 \\
              &   &        &  &\ddots &\vdots&\vdots \\
              &   &        &  &       & *& 0  \\
              &   &        &  &       &  & 0  \\
             \end{pmatrix}}}_{2n \times 2n}\in Y_{n,\,r}^2~,
\]
where $2\les i\les 2n-1.$  
Notice that the matrices in $Y_{n,\,r}^2$ are exactly the anti-transposes of those in 
$X_{n,\,r}^2$. By Lemma \ref{anti-diag-lemma} and Lemma \ref{prop-Xnr2} we obtain
\[\label{Ynr2}
    |Y_{n,\,r}^2|= (2n-2r)b_{n-1, \,r-1},\quad 1\les r\les n. \tag{3}
\]
When $n=1=r$ then $Y^2_{n,\,r}=\varnothing$, hence the formula is also valid when $n=1$.

We next calculate $|Y_{n,\,r}^3|$. If $f\in Y_{n,\,r}^3$, then $f(\bar{1})=1$ and $f$ is of the form
\[
\begin{pmatrix}
            0  & 0 & 0 &\cdots & 0& 1 \\
              & * & * &\cdots & *& 0  \\
              &   & * &\cdots & *& 0 \\
              &   &   &\ddots &\vdots&\vdots \\
              &   &   &       & *& 0  \\
              &   &   &       &  & 0  \\
             \end{pmatrix}_{2n \times 2n}\in Y_{n,\,r}^3~.
\]
A  similar argument to that in Lemma $\ref{lemma-Xnr1}$ yields that the map $f\mapsto f'$
is a bijection of $Y_{n,\,r}^3$ onto $\mathbb B_{n-1,\,r-1}$. Therefore,
\[\label{Ynr3}
    |Y_{n,\,r}^3|=b_{n-1, \,r-1},\quad 1\les r\les n. \tag{4}
\]

We now calculate $|Y_{n,\,r}^4|$. Recall
\[
Y_{n,\,r}^4 = \{f\in \mathbb B_{n,\,r}\mid f(i)=1 \text{ and } f(\bar{1})=j \text{ for some } 2\les i, j\les 2n-1\}.
\]
If $f\in Y_{n,\,r}^4$, then the $(1,1)$, $(1,2n)$, $(2n,2n)$ entries of $f$ are  $0$; hence $f$ has the form
\[
\begin{array}{l@{}l}
\hspace{1.75cm}{\overset{\text{\small{column} $i$}}{\downarrow}}&\\
\begin{pmatrix}
             0 & 0& \cdots &1 & \cdots& 0&\cdots & 0& 0 \\
              & * & \cdots &0 & \cdots& * &\cdots & *& 0 \\
              &   & \ddots   &\vdots&\cdots&\vdots&\cdots & \vdots& \vdots \\
              &   & &0 &\cdots &*&\cdots & *& 0 \\
              &   && &0 &0 & \cdots& 0& 1 \\
              &   &&& &* &\cdots & *& 0 \\
              &   & &&&  &\ddots &\vdots&\vdots \\
              &   & &&&  &       & *& 0  \\
              &   & & && &       &  & 0  \\
             \end{pmatrix}_{2n \times 2n} 
&\hspace*{-2em}\leftarrow\text{\small{row $j$}}\qq~ \in Y_{n,r}^4~,
\end{array}
\]
where  $2\les i,\, j\les 2n-1$. 

First, let $2\les r\les n$.
Deleting the first row, first column, last row, and last column of each $f\in Y_{n,\,r}^4$ produces an element $f'\in\mathbb B_{n-1, \,r-2}$. We have a map $\kappa\colon Y_{n,\,r}^4 \to \mathbb B_{n-1,\, r-2}$. 
We first show that for each $g\in \mathbb B_{n-1,\,r-2}$, there are $2(n-r+1)$ distinct $f$'s in $Y_{n,\,r}^4$ such that $f'=g$.  Indeed fix $g\in \mathbb B_{n-1,\,r-2}$.
Let $I=\tau^{-1}(D(g))$ and $J=\tau^{-1}(R(g))$, where $\tau(k)=k-1$.
By Lemma \ref{tau-lemma} a), $I$ and $J$ are admissible in $\mathscr A_n$, and $|I\cup \theta(I)|=|J\cup \theta(J)|=2r-4$, where $\theta$ is the involution 
in $\mb {2n}$. Let $E(g)$ and $F(g)$ denote the complements of $I\cup\theta(I)$ and $J\cup\theta(J)$ in 
${\bf 2n}\setminus \{1, 2n\}$, respectively, so that $|E(g)|=|F(g)|=2(n-r+1)\ges 2$.

For each  $i\in E(g)$ and $j\in F(g)$ we extend $g$ by adding a zero row at the bottom, a zero column at the left, and a row $1$ with $1$ in position $i$, and a column $2n$ with $1$ in position $j$; this extension is denoted by $\hat{g}_{ij}$. The domain
$D(\hat{g}_{ij})=\tau^{-1}(D(g))\cup\{i\}$ and the range
$R(\hat{g}_{ij})=\tau^{-1}(R(g))\cup\{ j\}$
are also admissible in $\mathscr A_n$ since both $\theta(D(\hat{g}_{ij}))\cap D(\hat{g}_{ij})$ and $\theta(R(\hat{g}_{ij}))\cap R(\hat{g}_{ij})$ are empty by the construction of $E(f)$, $E(g)$ and that $\theta(i)\ne i$, $\theta(j)\ne j$. Thus $\hat {g}_{ij}$ is in $Y_{n,\,r}^4$. Moreover $\kappa(\hat g_{ij})=g$ for all $(i,j)\in E(g)\times F(g)$.

Note that if $(i,j)\ne (i_1,j_1)$ in $E(g)\times F(g)$, then $\hat {g}_{ij}\ne\hat {g}_{i_1j_1}$.
Moreover, if $g\ne h$ in $\mathbb B_{n-1,\,r-2}$ then for each pair $(i,j)\in E(g) \times F(g)$ and $(k,l)\in E(h)\times F(h)$ we have
$\hat {g}_{ij}\ne\hat{h}_{k l}$.
Furthermore, if $f\in Y_{n,\,r}^4$ then by restriction we obtain $g=f'\in \mathbb B_{n-1,\,r-2}$
in which $f=\hat{g}_{ij}$ where $f(i)=1$ and $j=f(2n)$.
 
Thus $Y_{n,\,r}^4$ is a disjoint union of sets $\{\hat {g}_{ij}\mid i\in E(g), j\in F(g)\}$  as $g$ runs over $\mathbb B_{n-1,\,r-2}$. 
We obtain, for $2\les r\les n$
\[\label{Ynr4}
    |Y_{n,\,r}^4|=4(n-r+1)^2 \,b_{n-1, \,r-2}. \tag{5}
\] 
Now, let $n\ges 1$ and $r=1$. Then $Y_{n,\,r}^4=Y_{n,\,1}^4=\varnothing$; hence $|Y_{n,\,1}^4|=0$. Note that in this case $b_{n-1, \,r-2}=b_{n-1, \,-1}=0$. Thus formula \eqref{Ynr4} still holds.

Since $|Y_{n,\,r}| = |Y_{n,\,r}^1|+|Y_{n,\,r}^2|+|Y_{n,\,r}^3|+|Y_{n,\,r}^4|$, by (\ref{Ynr1}), (\ref{Ynr2}), (\ref{Ynr3}), and (\ref{Ynr4}) we obtain
\[\label{Ynr}
    |Y_{n,\,r}| = 2(n-r+1)b_{n-1,\, ~r-1} + 4(n-r+1)^2 b_{n-1,\, ~r-2},\q 1\les r\les n.\tag{6}
\]
It follows from (\ref{Xnr}) and (\ref{Ynr}) that
\begin{align*}{\label{bnr}}
  b_{n,\, r} &= |X_{n,\,r}| + |Y_{n,\,r}|\\
            &= b_{n-1, \,r} + (4n-4r+3) b_{n-1, ~r-1} + 4(n-r+1)^2 b_{n-1,\, ~r-2},\q 1\les r\les n.\tag{7}
\end{align*}
Moreover when $n\ges 1$ and $r=0$ equality (\ref{bnr}) still holds since $b_{0,\,0}=1$ and $b_{n,\, r} = 0$ if $n<r$ or $r<0$,  so that both sides are equal to $1$. The proof of 1) is complete.

We can now prove 2). Recall that
$
   \mathbb{B}_n = \big(\bigsqcup_{r=0}^n \mathbb{B}_{n,\, r}\big)\bigsqcup\,\{\id_{2n}\}.
$ 
From (\ref{bnr}) we  obtain
\begin{align*}
    b_n &=1+ \sum_{r=0}^{n-1} b_{n-1, \,r} + 3\sum_{r=1}^{n}b_{n-1,\,r-1} +4\sum_{r=1}^{n}(n-r) b_{n-1, ~r-1} + 4\sum_{r=2}^{n}(n-r+1)^2 b_{n-1, \,r-2}  \\
        &=4\Big(1+\sum_{r=0}^{n-1}b_{n-1,\,r}\Big) +4\sum_{r=1}^{n}(n-r) b_{n-1, \,r-1} + 4\Big[\sum_{r=2}^{n}(n-r+1)^2 b_{n-1, \,r-2}\Big]-3\\
        &= 4b_{n-1} + 4\sum_{r=1}^{n-1} (n-r)b_{n-1, ~r-1} + 4\Big[\sum_{r=1}^{n-1} (n-r)^2 b_{n-1, \,r-1}\Big]-3 \\
        &= 4b_{n-1} + 4\Big[\sum_{r=1}^{n-1} (n-r)(n-r+1)b_{n-1, \,r-1}\Big] - 3~,\quad\mbox{for } n\ges 1.
\end{align*}
The proof of 2) is complete.

We now prove 3). From 2) it follows that
\begin{align*}
b_n&=4b_{n-1}+\Big[\sum_{r=1}^{n-1}4(n-r+1)(n-r)b_{n-1,\,r-1}\Big]-3 \\
&=4\Big(1+\sum_{r=0}^{n-1}4b_{n-1,\,r}\Big) +\Big[\sum_{r=0}^{n-2}4(n-r)(n-r-1)b_{n-1,\,r}\Big]-3\\
&=1+4b_{n-1,\,n-1}+\sum_{r=0}^{n-2}4b_{n-1,\,r}+\sum_{r=0}^{n-2}4(n-r)(n-r-1)b_{n-1,\,r}\\
&=1+4b_{n-1,\,n-1}+\sum_{r=0}^{n-2}4\big[1+(n-r)(n-r-1)  \big]     b_{n-1,\,r}\\
&=1+\sum_{r=0}^{n-1}4\big[1+(n-r)(n-r-1)  \big]    b_{n-1,\,r}
\end{align*}
where the last equality follows from $1+(n-r)(n-r-1)  =1$ when $r=n-1$.  
\end{proof}
The first seven cases of $b_{n, \,r}$ and $b_{n}$ are provided in the following table.
\begin{table}[h]
\centering
\renewcommand{\arraystretch}{1.2} 
\begin{tabular}{|r|rrrrrrrr|r||r|}
\hline
$b_{n, r}$ &r: 0 & 1 & 2 & 3 & 4 & 5 & 6&7&\; 2n&$b_n$\\ \hline
n:1                          & 1  & 3  & 1  &  &&&&& &5 \\
2                          & 1  & 10  &  13 &  &1&&&& &25 \\
3                          & 1  & 21  & 99  & 79 &&&1&& &201\\
4                          & 1  & 36  & 366  & 1108 &633&&&&1 &2145\\ 
5                          & 1  & 55  & 970  & 6430  & 14245  &6331    &&&1&28033\\ 
6                          & 1  & 78  & 2115 & 24500 & 119895 & 208926 & 75973 &&1& 431489\\ 
7                          & 1  & 105 & 4053 & 72485 & 622755 & 2409771& 3456775 & 1063623 &1& 7629569\\ \hline
\end{tabular}
\end{table}

\section{Generating Functions and Closed Formulas}

\subsection{The Exponential Generating Function of $b_n$}
We investigate the exponential generating function of $b_n$. Recall that $b_n=|\mathbb B_n|$ for $n\ges 1$, and $b_{0}=2$.
\begin{theorem}{\label{generatingBn}}
The series $\sum_{n=0}^{\infty}(b_n/n!)t^n$ converges for $|t|<1/2$ to the function 
\[b(t)=e^t\Big[1 + \frac{e^{t/(1-2t)}}{1-2t}\Big].\]
\end{theorem}
\begin{proof} 
We first investigate the exponential generating function of $s_n:=b_n-1$, where $n\ge 0$. Then, 
\[
    s_n = \sum_{r=0}^{n} b_{n,\, r},\q\text{with}\q s_0 = 1.
\]
By Theorem \ref{mainTheorem} part 1), if $1\les n$ and $0\les r\les n$ then 
\begin{align*}
b_{n,\,r}=b_{n-1,\,r}+(4n-4r+3)b_{n-1,\,r-1}+4(n-r+1)^2b_{n-1,\,r-2}\;,
\end{align*}
where $b_{0,\,0}=1$ and $b_{j,\,k}=0$ for $j<k$ or $k<0$.

Let $s(t)=\sum_{n=0}^{\infty}(s_n/n!)t^n$. 
We claim that $s(t)$ converges for $|t|<1/2$ to the function 
\[\frac{e^te^{t/(1-2t)}}{1-2t}.\]
Indeed, write $s_{n,\,k}=b_{n,\,n-k}$ for $0\les k\les n$. Then $b_{n,\,r}=s_{n,\,n-r}$,
$s_{0,\,0}=1$, and $s_{j,\,k}=0$ for $j<k$ or $k<0$. Thus
\begin{align*}{\label{snk}}
s_{n,\,k}&=s_{n-1,\,k-1}+(4k+3)s_{n-1,\,k}+4(k+1)^2s_{n-1,\,k+1},  \q 1\les n; 0\les k\les n.\tag{8}
\end{align*}
Multiplying by $x^k$ and summing, we obtain
\[{\label{sumOfSnkx}}
\sum_{k=0}^ns_{n,\,k}\;x^k=\sum_{k=0}^ns_{n-1,\,k-1}\;x^k+\sum_{k=0}^n(4k+3)s_{n-1,\,k}\;x^k
+\sum_{k=0}^n4(k+1)^2s_{n-1,\,k+1}\;x^k.  \tag{9}
\]
Define polynomials
\[{\label{SumOfSnkx}}
S_n(x)=\sum_{k=0}^n s_{n,\,k}\;x^k,\quad n\ges 0.\tag{10}
\]
Then
\[
S_{n-1}(x)=\sum_{k=0}^{n-1}s_{n-1,\,k}\;x^k;\;
S_{n-1}'(x)=\sum_{k=1}^{n-1}ks_{n-1,\,k}\;x^{k-1};\;
S_{n-1}''(x)=\sum_{k=2}^{n-1}k(k-1)s_{n-1,\,k}\;x^{k-2}.
\]
The three sums on the right of \eqref{sumOfSnkx} are
\begin{align*}
&\sum_{k=0}^ns_{n-1,\,k-1}\;x^k=x\sum_{k=-1}^{n-1}s_{n-1,\,k}\;x^k=xS_{n-1}(x)\,,\\
&\sum_{k=0}^n(4k+3)s_{n-1,\,k}\;x^k=\sum_{k=0}^{n-1}(4k+3)s_{n-1,\,k}\;x^k=4xS'_{n-1}(x)+3S_{n-1}(x)\,,\\
&\sum_{k=0}^n4(k+1)^2s_{n-1,\,k+1}\;x^k
=\sum_{k=1}^{n+1}4k^2s_{n-1,\,k}\;x^{k-1}\\
&\hspace{4.1cm}=4x\sum_{k=1}^{n-1}k(k-1)s_{n-1,\,k}\;x^{k-2}
+4\sum_{k=1}^{n-1}ks_{n-1,\,k}\;x^{k-1}&&{}\\
&\hspace{4.1cm}=4xS''_{n-1}(x)+4S_{n-1}'(x)\,.
\end{align*}
Therefore, \eqref{sumOfSnkx} becomes
\[
S_n(x)=4xS_{n-1}''(x)+(4x+4)S'_{n-1}(x)+(x+3)S_{n-1}(x), \q n\ges 1.
\]
We have
\[
\sum_{n=1}^\infty \frac{S_n(x)t^{n-1}}{(n-1)!}=4x\sum_{n=1}^\infty \frac{S_{n-1}''(x)t^{n-1}}{(n-1)!}+(4x+4)\sum_{n=1}^\infty \frac{S'_{n-1}(x)t^{n-1}}{(n-1)!}+(x+3)\sum_{n=1}^\infty \frac{S_{n-1}(x)t^{n-1}}{(n-1)!}.
\]
Define the function
\[{\label{uxt}}
    u(x,\,t)=\sum_{n=0}^\infty \frac{S_n(x)t^n}{n!}\tag{11}
\]
so that
$u(x,\,0)=S_0(x)=s_{0,\,0}=1$.  
We obtain a PDE with an auxiliary condition:
\[
\frac{\partial u}{\partial t}=4x \frac{\partial u^2}{\partial x^2}+(4x+4)\frac{\partial u}{\partial x}
+(x+3)u,\q u(x,\,0)=1.
\]
We seek a solution in the form $u(x,\,t)=\beta(t) e^{x\alpha(t)}$.
The auxiliary condition gives 
$$1=\beta(0)e^{x\alpha(0)} \q\text{for all } x.
$$ 
Hence $\alpha(0)=0$ and $\beta(0)=1$.
Substituting the ansatz in the differential equation, we find
\begin{align*}
e^{x\alpha}(\beta'+x\beta\alpha')=4x\alpha^2\beta e^{x\alpha} +(4x+4)\alpha\beta e^{x\alpha} 
+(x+3)\beta e^{x\alpha}.
\end{align*}
Equating coefficients of $xe^{x\alpha}$ and $e^{x\alpha}$ from both sides, we obtain two equations, namely
 \begin{align*}
 \alpha'&=(2\alpha+1)^2,\q \alpha(0)=0\,;\\[4pt]
 \beta'&=(4\alpha+3)\beta,\q\beta(0)=1\,.
 \end{align*}
The solutions are
\begin{align*}
&\alpha=\frac{t}{1-2t}\q\text{and}\q\beta=\frac{e^t}{1-2t},\q t<\tfrac12.
\end{align*}
Hence
\[{\label{uxtResult}}
u(x,\,t)=\frac{e^te^{tx/(1-2t)}}{1-2t},\q |t|<\tfrac12,~ x\in\R. \tag{12}
\]
Note that from \eqref{SumOfSnkx} we have
\begin{align*}
S_n(1)&=\sum_{k=0}^n s_{n,\,k}=\sum_{r=0}^nb_{n,\,r}=s_n,\q n\ges 0.
\end{align*}
From \eqref{uxt} and \eqref{uxtResult} it follows that
\begin{align*}{\label{stResult}}
s(t)&=\sum_{n=0}^\infty\frac{ s_n\,t^n}{n!}=u(1,\,t)=\frac{e^te^{t/(1-2t)}}{1-2t},\tag{13}
\end{align*}
and the series is convergent for $|t|<\frac12$. The claim holds.

Next, note that
$
b_n=1+s_n.
$
We obtain
\begin{align*}
  b(t)&=\sum_{n=0}^{\infty}b_n\frac{t^n}{n!}=\sum_{n=0}^{\infty}\frac{t^n}{n!}+\sum_{n=0}^{\infty}s_n\frac{t^n}{n!}=e^t+s(t)\\
    &=e^t + \frac{e^te^{t/(1-2t)}}{1-2t}~.
\end{align*}
The desired result follows.
\end{proof}

\subsection{An Explicit Closed Formula for $b_n$}

\begin{theorem} {\label{formula:bn}}
Let $b_n=|\mathbb B_n|$ for $n\ges 1$. Then
\[
b_n=1+\sum_{k=0}^n\Bigg\{\sum_{l=0}^k\binom{n}{k}
\binom{k}{l}\frac{ 2^k\,k!}{2^l\,\;l!}\,\Bigg\}\,,\q n\ges 1\,.
\]
\end{theorem}
\begin{proof}
  The Laguerre polynomials are 
\[
L_n(x)=\frac{e^x}{n!}\frac{d^n}{dx^n}(x^ne^{-x})=\sum_{k=0}^n\frac{n!(-x)^k}{(k!)^2(n-k)!}
\,,\q x\in\R,\; n\ges 0.
\]
Their exponential generating function is
\begin{align*}
\sum_{n=0}^\infty L_n(x)z^n&=\frac{e^{-xz/(1-z)}}{1-z},\q x\in\R,~ |z|<1.
\end{align*}
Substituting $x=-\frac12$ and $z=2t$, we obtain
\[
\sum_{n=0}^\infty L_n(-\tfrac12)(2t)^n=\frac{e^{t/(1-2t)}}{1-2t},\q|t|<\tfrac12.
\]
Substituting this in \eqref{stResult}, we have
\begin{align*}{\label{stLaguerre}}
\sum_{n=0}^\infty \frac{s_nt^n}{n!}=e^t\sum_{n=0}^\infty L_n(-\tfrac12)2^nt^n
&= \bigg( \sum_{n=0}^\infty\frac{t^n}{n!}\bigg)\bigg( \sum_{n=0}^\infty L_n(-\tfrac12)2^nt^n\bigg)\\
&=\sum_{n=0}^\infty \bigg\{\sum_{k=0}^n \frac{L_k(-\tfrac12)2^k}{(n-k)!}\bigg\} t^n \tag{14} \\
&=\sum_{n=0}^\infty \bigg\{\sum_{k=0}^n\sum_{l=0}^k\frac{k!2^{k-l}}{(l!)^2(n-k)!(k-l)!}  \bigg\}  t^n.
\end{align*}
Hence 
\[
s_n=
\sum_{k=0}^n\sum_{l=0}^k\frac{n!k!2^{k-l}}{(l!)^2(n-k)!(k-l)!}   
=\sum_{k=0}^n\bigg\{\sum_{l=0}^k\binom{n}{k}\binom{k}{l}\frac{2^k k!}{2^ll!}\bigg\}\,.
\]
As $b_n=1+s_n$ we obtain our result.
\end{proof}

The order $b_n$ of the upper triangular symplectic Renner monoids is closely related to the Laguerre polynomials.
\begin{corollary}{\label{coro:LaguerrePolynomial}}
Let $L_k(t)$ be the $k$th Laguerre polynomial. Then
$$
    b_n = 1+n!\sum_{k=0}^n\frac{2^k}{(n-k)!}L_k\left(-\tfrac12\right),\q n\ges 0.
$$
\end{corollary}
\begin{proof}
  The proof follows from comparing the coefficients of $t^n$ on both sides of \eqref{stLaguerre}.
\end{proof}

We shall obtain a recurrence relation for the sequence $b_n$ involving its immediate past three terms: $b_{n-1}$, $b_{n-2}$, and $b_{n-3}$.
\begin{theorem} {\label{th:pureRecursion}}
Let  $b_n=1+\sum\limits_{r=0}^nb_{n,r}$, where $n\ges 0$.  Then 
$b_0=2$, $b_1=5$, $b_2=25$ and
\[
b_n=4nb_{n-1}-4n(n-1)b_{n-2}+4(n-1)(n-2)b_{n-3}+4n-7,\q n\ges 3.
\]

\end{theorem}
\begin{proof} 
Let $s_n=\sum_{r=0}^nb_{n,r}$. From \eqref{stResult} it follows that
\begin{align*}
s(t):=\sum_{n=0}^\infty\frac{ s_n\,t^n}{n!}=\frac{e^te^{t/(1-2t)}}{1-2t},\q |t|<\tfrac12.
\end{align*}
A straightforward calculation shows that
\begin{align*}
 \frac{s'(t)}{s(t)}=\frac{4(1-t)^2}{(1-2t)^2}.
\end{align*}
That is,
\[
(4t^2-4t+1)s'(t)=(4t^2-8t+4)s(t).
\]
The left side is
\begin{align*}
4\sum_{n=2}^\infty &\frac{ n(n-1)s_{n-1}t^n}{n!}-4\sum_{n=1}^\infty\frac{ns_nt^n}{n!}
+\sum_{n=1}^\infty\frac{s_{n+1}t^n}{n!}.
\end{align*}
The right side is
\begin{align*}
4\sum_{n=0}^\infty &\frac{ s_{n}t^{n+2}}{n!}-8\sum_{n=0}^\infty\frac{s_nt^{n+1}}{n!}
+4\sum_{n=0}^\infty\frac{s_{n}t^n}{n!}\\[6pt]
&=4\sum_{n=2}^\infty \frac{ s_{n-2}t^{n}}{(n-2)!}-8\sum_{n=1}^\infty\frac{s_{n-1}t^{n}}{(n-1 )!}
+4\sum_{n=0}^\infty\frac{s_{n}t^n}{n!}\\[6pt]
&=4\sum_{n=2}^\infty \frac{ n(n-1)s_{n-2}t^{n}}{n!}-8\sum_{n=1}^\infty\frac{ns_{n-1}t^{n}}{n!}
+4\sum_{n=0}^\infty\frac{s_{n}t^n}{n!}.
\end{align*}
Comparing coefficients of $t^0$, $t$ gives $s_1=4s_0=4$ and 
$-4s_1+s_2=-8s_0+4s_1$ so $ s_2=24$.
Whereas, for $n\ges 2$ we have
\begin{align*}
4n(n-1)s_{n-1}-4ns_n+s_{n+1}=4n(n-1)s_{n-2}-8ns_{n-1}+4s_n
\end{align*}
which rearranges to 
\[
s_n=4ns_{n-1}-4n(n-1)s_{n-2}+4(n-1)(n-2)s_{n-3},\q n\ges 3. 
\]
Since $b_n=s_n+1$, we obtain
$b_0=2$, $b_1=5$, $b_2=25$ and
\[
b_n=4nb_{n-1}-4n(n-1)b_{n-2}+4(n-1)(n-2)b_{n-3}+4n-7,\q n\ges 3.\qedhere
\]
\end{proof}

\subsection{The Exponential Generating Function of \(b_{n,\,n-r}\)}
We now fix an $r\ges 0$, and consider the exponential generating function for the sequence $n\mapsto b_{n,n-r}$. We note that $\sum\limits_{n=0}^\infty  \frac{b_{n,n-r\,}\,t^n}{n!}=\sum\limits_{n=r}^\infty  \frac{b_{n,n-r\,}\,t^n}{n!}$.
\begin{theorem}{\label{generatingBnr}}
Let $n\ges 1$ and let $b_{n,r}=|\mathbb B_{n,r}|$ for $0\les r\les n$. Then 
\[
\sum_{n=r}^\infty  \frac{b_{n,n-r\,}\,t^n}{n!}=\frac{t^re^t}{r!(1-2t)^{r+1}},\qq |t|<\tfrac12,~ r\ges 0.
\]
\end{theorem}
\begin{proof}
Fix $r\ges 0$.
Write $ s_{n,\,r}=b_{n,\,n-r}$ so that
\[\label{brt}
   \sum_{n=r}^\infty \frac{s_{n,\,r\,\,}t^n}{n!}
    = \sum_{n=r}^\infty \frac{ b_{n,n-r\,}\,t^n}{n!}.\tag{15}
\]
By \eqref{SumOfSnkx}, \eqref{uxt} and \eqref{uxtResult}, and the series expansion for $\exp x$ we have
\begin{align*}
\sum_{n=0}^\infty \Big\{\sum_{r=0}^n \frac{s_{n,r\,}x^rt^n}{n!}\Big\}&=\frac{e^te^{tx/(1-2t)}}{1-2t}
        =\frac{e^t}{1-2t} \sum_{r=0}^{\infty}\frac{t^rx^r}{r!(1-2t)^r}. \tag{16}
\end{align*}
An interchange of order of summation and (15) gives
\begin{align*}
\sum_{n=0}^\infty \Big\{\sum_{r=0}^n \frac{s_{n,r\,}x^rt^n}{n!}\Big\}
=\sum_{r=0}^\infty\Big\{ \sum_{n=r}^\infty\frac{s_{n,r\,\,}x^rt^n}{n!} \Big\}
=\sum_{r=0}^\infty\sum_{n=r}^\infty \frac{ b_{n,n-r\,}\,x^rt^n}{n!}.
\end{align*}
Hence
\[
\sum_{r=0}^\infty\Big\{\sum_{n=r}^\infty \frac{ b_{n,n-r\,}\,t^n}{n!}\Big\} x^r
=\frac{e^t}{1-2t} \sum_{r=0}^{\infty}\frac{t^r}{r!(1-2t)^r}\,x^r.
\]
Comparing coefficients of $x^r$, we obtain
\[ \label{snkGenerating}
    \sum_{n=r}^\infty  \frac{b_{n,n-r\,}\,t^n}{n!} = \frac{t^r e^t} {r!(1-2t)^{r+1}},\q |t|<\tfrac12,~ r\ge 0. \tag{17}
\]
The proof is complete.
\end{proof}

\subsection{An Explicit Closed Form for Every \(b_{n,\,r}\)}

\begin{theorem}{\label{formula:bnr}}
Let $n\ges 1$ and let $0\les r\les n$. Then
\[
b_{n,r}= \sum_{k=0}^r\frac{n!\,2^k}{(n-r)!(r-k)!}\binom{n-r+k}{k}.
\]
\end{theorem}
\begin{proof}
Fix $r\ges 0$. By \eqref{snkGenerating} 
\[
\sum_{n=r}^\infty  \frac{b_{n,n-r\,}\,t^n}{n!}=\frac{t^re^t}{r!(1-2t)^{r+1}}.
\]
Fix $N$, $R$ with $N\ges 1$ and $0\les R\les N$. Put $r=N-R$; we obtain
\[\label{nnNr}
\sum_{n=N-R}^\infty  \frac{b_{n,n-N+R\,}\,t^n}{n!}=\frac{t^Nt^{-R}e^t}{(N-R)!(1-2t)^{N-R+1}}.\tag{18}
\]
In the left side the coefficient of $t^N$ is
\[\dfrac{b_{N,R}}{N!}.\]
Put $\alpha=N-R+1$, so $\alpha$  is a positive integer. We have expansions
\[
e^t=\sum_{n=0}^\infty \frac{t^n}{n!},\text{\q and \q}
(1-2t)^{-\alpha}=\sum_{n=0}^\infty \binom{-\alpha} {n}(-2t)^n.
\]
Hence 
\[e^t(1-2t)^{-\alpha}=\sum_{n=0}^\infty c_n t^n \text{\q where\q}
c_n=\sum_{k=0}^n\frac1{(n-k)!}\binom{-\alpha} {k}(-1)^k2^k.
\]
It can easily be shown that
\begin{align*}
\binom{-\alpha} {k}(-1)^k
=\binom{N-R+k}{k}.
\end{align*}
The right side of \eqref{nnNr}  is 
\[
\frac{t^Nt^{-R}}{(N-R)!}\{c_0+c_1t+\dots+c_Rt^R+\dots\}.
\]
Hence
\[
\frac{b_{N,R}}{N!}=\frac{c_R}{(N-R)!}=\frac{1}{(N-R)!}
\sum_{k=0}^R\frac{2^k}{(R-k)!}\binom{N-R+k}{k}.\qedhere
\]
\end{proof}


\end{document}